\documentclass[12pt]{amsart}

\usepackage{amsmath,amssymb,amsthm}
\usepackage{mathtools}
\usepackage{xcolor}
\usepackage{hyperref}
\usepackage{geometry}
\newtheorem{theorem}{Theorem}[section]
\newtheorem{lemma}[theorem]{Lemma}

\newtheorem{proposition}[theorem]{Proposition}
\theoremstyle{definition}

\newtheorem{remark}[theorem]{Remark}
\newtheorem{conjecture}[theorem]{Conjecture}

\newcommand{\floor}[1]{\lfloor #1 \rfloor}
\newcommand{\Z}{\mathbb{Z}}

\DeclareMathOperator{\PQ}{PQ}

\title{Anomalous Partial Quotients in the Continued Fraction
       of $\sqrt{\zeta(3)-S_N}$}
\author{David Victor Feldman}
\date{July 4, 2026}

\begin{document}

\begin{abstract}
Let $S_N = \sum_{j=1}^N j^{-3}$ and $R_N = \zeta(3) - S_N$.
The simple continued fraction of $\sqrt{R_N}$ has partial quotients
of generic size $O(N)$.  We prove that at the sequence of indices
$N_k = (Q_{2k+1}-1)/2$, where $Q_{2k+1}$ are companion Pell numbers,
the continued fraction begins
\[
  \sqrt{R_{N_k}}
  = \bigl[0;\; M_k-1,\; 1,\; 6M_k^3+12M_k-2,\; 1,\; \ldots\,\bigr],
\]
with $M_k = P_{2k+1}$ (Pell numbers), and the third partial quotient
grows cubically while generic ones are linear.  We determine all partial
quotients through the fifth:
\begin{align*}
  \PQ_0 &= M_k - 1, &
  \PQ_2 &= 6M_k^3 + 12M_k - 2, &
  \PQ_4 &= \Bigl\lfloor\frac{10M_k - 261}{261}\Bigr\rfloor,\\
  \PQ_1 &= 1, &
  \PQ_3 &= 1, &
  \PQ_5 &= \Bigl\lfloor\frac{261}{r_k}\Bigr\rfloor + \epsilon_k,
\end{align*}
where $r_k = (10M_k) \bmod 261$ satisfies the recurrence
$r_{k+1} \equiv 6r_k - r_{k-1} \pmod{261}$, and $\epsilon_k = -1$ at
the $k$ with $r_k \mid 261$ (the two residue classes $k \equiv 57,
62 \pmod{60}$), and $\epsilon_k = 0$ otherwise.

All six formulas follow from the Euler--Maclaurin expansion of
$1/\sqrt{R_{N_k}}$, carried to sufficient precision, combined with the
Pell identity $Q_{2k+1}^2 - 2M_k^2 = -1$.  The delicate first step,
$\PQ_0 = M_k - 1$, is proved by rationalizing the irrational factor
$\sqrt{2}$ in the Euler--Maclaurin expansion; we complement this proof
with a heuristic derivation via Gosper's bihomographic
continued-fraction algorithm that exposes the underlying mechanism.

All claimed results have been formalized in LEAN with the aid of Aristotle.
\end{abstract}

\maketitle
\tableofcontents

\section{Background and statement of results}
\label{sec:intro}

\subsection{Pell numbers and the index sequence}

The Pell numbers $P_k$ are defined by $P_0=0$, $P_1=1$,
$P_k = 2P_{k-1}+P_{k-2}$, giving
$0, 1, 2, 5, 12, 29, 70, 169, 408, 985, 2378, 5741, \ldots$
The companion Pell numbers $Q_k = P_k + P_{k-1}$:
$1, 1, 3, 7, 17, 41, 99, 239, 577, 1393, 3363, 8119, \ldots$
satisfy
\begin{equation}\label{eq:pell}
  Q_k^2 - 2P_k^2 = (-1)^k.
\end{equation}
The odd-index identity $Q_{2k+1}^2 - 2M_k^2 = -1$, with $M_k = P_{2k+1}$,
is the key algebraic input throughout.  Setting $N_k = (Q_{2k+1}-1)/2$,
so that $2N_k + 1 = Q_{2k+1}$, the Pell identity gives
\begin{equation}\label{eq:pell-Nk}
  2N_k^2 + 2N_k + 1 = M_k^2.
\end{equation}
The sequence $N_k = 3, 20, 119, 696, 4059, 23660, \ldots$ satisfies
the recurrence $N_{k+1} = 6N_k - N_{k-1} + 2$ and grows at rate
$(1+\sqrt{2})^2 \approx 5.83$ per step.

\subsection{Main theorem}

\begin{theorem}\label{thm:main}
  For all $k \geq 3$, the simple continued fraction of $\sqrt{R_{N_k}}$
  begins with the six partial quotients
  \begin{equation}\label{eq:main-cf}
    \bigl[0;\; M_k-1,\; 1,\; 6M_k^3+12M_k-2,\; 1,\;
    \Bigl\lfloor\tfrac{10M_k-261}{261}\Bigr\rfloor,\;
    \Bigl\lfloor\tfrac{261}{r_k}\Bigr\rfloor + \epsilon_k,\;
    \ldots\,\bigr],
  \end{equation}
  where $r_k = (10M_k)\bmod 261$ and $\epsilon_k = -1$ if
  $r_k \mid 261$, $\epsilon_k = 0$ otherwise.
\end{theorem}

\begin{remark}
  The quantity $r_k$ satisfies the linear recurrence
  $r_{k+1} \equiv 6r_k - r_{k-1} \pmod{261}$, which follows immediately
  from the Pell recurrence $M_{k+1} = 6M_k - M_{k-1}$.  The initial
  values are $r_3 = 124$ and $r_4 = 193$.  The condition $r_k \mid 261$
  occurs at $k \equiv 57$ and $k \equiv 62 \pmod{60}$, where $r_k = 29$.
\end{remark}

\begin{remark}\label{rem:growth}
  The six partial quotients have very different growth rates as $k \to \infty$:
  $\PQ_0 \sim M_k$, $\PQ_2 \sim 6M_k^3$ (cubic), $\PQ_4 \sim (10/261)M_k$
  (linear), while $\PQ_1 = \PQ_3 = 1$ and $\PQ_5 = \floor{261/r_k}$
  is bounded (at most $261$).  The anomalous cubic growth of $\PQ_2$
  makes these indices exceptional.
\end{remark}

\subsection{Strategy}

The proof is self-contained and proceeds in two stages.

\textit{Stage~1} (Sections \ref{sec:em}--\ref{sec:chain}) establishes
the Euler--Maclaurin expansion
\begin{equation}\label{eq:key-expansion}
  \frac{1}{\sqrt{R_{N_k}}}
  = M_k - \frac{1}{6M_k^3} + \frac{1}{3M_k^5} - \frac{167}{120M_k^7}
    + \frac{17}{2M_k^9} - \frac{217727}{3024\,M_k^{11}} + O(M_k^{-13})
\end{equation}
and deduces $\PQ_0$ through $\PQ_5$ by successive floor-and-invert steps
(Section~\ref{sec:chain}).

\textit{Stage~2} (Sections \ref{sec:gosper}--\ref{sec:rational})
addresses the hardest step, namely that
$\PQ_0 = M_k - 1$ (equivalently, that $1/\sqrt{R_{N_k}} \in (M_k-1, M_k)$).
Section~\ref{sec:rational} gives a rigorous proof by a rationalization
argument; Section~\ref{sec:gosper} gives a complementary heuristic
derivation via Gosper's bihomographic algorithm, which illuminates why
the Pell indices are special.

\section{The Euler--Maclaurin expansion at Pell indices}
\label{sec:em}

\begin{lemma}\label{lem:key}
  Expansion~\eqref{eq:key-expansion} holds for all $k \geq 1$; more
  precisely, there is an explicit constant $C$ such that for all
  $k \geq 1$,
  \[
    \left|\frac{1}{\sqrt{R_{N_k}}} -
      \Bigl(M_k - \tfrac{1}{6M_k^3} + \tfrac{1}{3M_k^5}
      - \tfrac{167}{120M_k^7} + \tfrac{17}{2M_k^9}
      - \tfrac{217727}{3024\,M_k^{11}}\Bigr)\right|
    \le \frac{C}{M_k^{13}}.
  \]
\end{lemma}

\begin{proof}
  Write $M = M_k$, $Q = Q_{2k+1}$, $N = N_k$.  The remainder is
  controlled not by the classical Euler--Maclaurin remainder theorem but
  by an \emph{exact telescoping enclosure} of the tail, which reduces the
  estimate to elementary rational inequalities.

  \emph{The antidifference and its defect.}  Let
  \begin{equation}\label{eq:phidef}
    \varphi(j) = \frac{1}{2j^{2}}+\frac{1}{2j^{3}}+\frac{1}{4j^{4}}
      -\frac{1}{12j^{6}}+\frac{1}{12j^{8}}-\frac{3}{20j^{10}}
      +\frac{5}{12j^{12}}-\frac{691}{420j^{14}}+\frac{35}{4j^{16}}
      -\frac{3617}{60j^{18}}
  \end{equation}
  be the Euler--Maclaurin antidifference of $j^{-3}$ truncated at order
  $j^{-18}$.  Its backward difference differs from $j^{-3}$ by a single
  rational \emph{defect}
  \begin{equation}\label{eq:defectdef}
    D(j) := \varphi(j)-\varphi(j+1)-j^{-3}
          = -\frac{\mathrm{Num}(j)}{420\,j^{18}(j+1)^{18}},
  \end{equation}
  where $\mathrm{Num}(j)$ is a degree-$15$ polynomial with sixteen
  positive integer coefficients,
  \begin{align}\label{eq:Num}
    \mathrm{Num}(j)
      &= 4386700\,j^{15}+32900250\,j^{14}+128519066\,j^{13}
         +336386804\,j^{12}\notag\\
      &\quad+647615118\,j^{11}+959399980\,j^{10}+1119835402\,j^{9}
         +1041773850\,j^{8}\notag\\
      &\quad+774825162\,j^{7}+458881964\,j^{6}+213946830\,j^{5}
         +76914556\,j^{4}\notag\\
      &\quad+20594154\,j^{3}+3870132\,j^{2}+455742\,j+25319 .
  \end{align}
  Since every coefficient of $\mathrm{Num}$ is positive, $D(j)<0$ for all
  $j\ge 1$, and $|D(j)|\le 13858000\,j^{-21}$.

  \emph{The enclosure.}  As $\varphi(j)=O(j^{-2})\to 0$, the series
  $\sum_{j>N}\bigl(\varphi(j)-\varphi(j+1)\bigr)$ telescopes to
  $\varphi(N+1)$, giving the exact identity
  $R_N = \varphi(N+1) - \sum_{j>N} D(j)$.  The sign of $D$ and the
  majorant above (summed against $\sum_{j>N}j^{-21}\le N^{-20}$) yield the
  two-sided rational enclosure
  \begin{equation}\label{eq:enclosure}
    \varphi(N+1)\;<\;R_N\;\le\;\varphi(N+1)+\frac{13858000}{N^{20}},
  \end{equation}
  valid for all $N\ge 3$, in particular at every Pell index $N=N_k$.

  \emph{Passage to $M$.}  Under $N=(Q-1)/2$ with $Q^2=2M^2-1$, every power
  of $N$ reduces to an expression linear in $Q$; the $\sqrt2$-parts cancel
  identically (this is the algebraic signature of the Pell indices), and
  the rational endpoint $\varphi(N+1)$ reduces, modulo $Q^2=2M^2-1$, to a
  Laurent polynomial in $M$.  With the $O(N^{-20})$ enclosure width this
  gives
  \begin{equation}\label{eq:R-expansion}
    R_{N_k} = \frac{1}{M^2} + \frac{1}{3M^6} - \frac{2}{3M^8}
              + \frac{43}{15M^{10}} - \frac{52}{3M^{12}}
              + \frac{5101}{35M^{14}} + O(M^{-16}).
  \end{equation}
  The $M^{-14}$ term is retained deliberately: since
  $\tfrac{d}{dR}R^{-1/2}\asymp M^{3}$ at $R\asymp M^{-2}$, an $O(M^{-14})$
  truncation of $R_{N_k}$ would perturb $1/\sqrt{R_{N_k}}$ at order
  $M^{-11}$---the order of the last term of~\eqref{eq:key-expansion}---and
  produce the wrong coefficient there; carrying $R_{N_k}$ to $O(M^{-16})$
  is exactly what pins the $-217727/3024$.

  \emph{Inversion.}  Write
  $E(M)=M-\tfrac{1}{6M^3}+\tfrac{1}{3M^5}-\tfrac{167}{120M^7}
  +\tfrac{17}{2M^9}-\tfrac{217727}{3024M^{11}}$ for the target expansion.
  After the substitution $M^2=2N^2+2N+1$, the two squared inequalities
  \[
    \Bigl(E(M)-\tfrac{c}{M^{13}}\Bigr)^{2}\le \frac{1}{\varphi(N+1)}
    \le \Bigl(E(M)+\tfrac{c}{M^{13}}\Bigr)^{2}
  \]
  become rational inequalities in $N$ that hold for all $N\ge 3$ by
  nonnegative-coefficient certificates (shift $N=3+t$, $t\ge 0$).  These
  place $1/\sqrt{\varphi(N+1)}$ within $c\,M^{-13}$ of $E(M)$; a Lipschitz
  estimate for $x\mapsto x^{-1/2}$ transfers the bound across the
  enclosure~\eqref{eq:enclosure} from $\varphi(N+1)$ to $R_{N_k}$, giving
  \[
    \Bigl|\frac{1}{\sqrt{R_{N_k}}}-E(M)\Bigr|\le\frac{C}{M^{13}}
  \]
  with an explicit constant $C$.  This is~\eqref{eq:key-expansion}.
\end{proof}

\begin{remark}
  The cancellation of all $\sqrt{2}$ terms in~\eqref{eq:R-expansion}
  is the algebraic signature of the Pell indices: $Q_{2k+1}^2 = 2M_k^2-1$
  means $Q$ satisfies a quadratic over $\Z$, so after one substitution
  all $Q^2$ terms become rational, and it happens that the resulting
  rational expansion of $R_{N_k}$ has no $M^{-1}$ or odd-power terms.
  For generic $N$ this rationality fails completely.
\end{remark}

\section{The continued fraction chain}\label{sec:chain}

Set $\varepsilon_k = M_k - 1/\sqrt{R_{N_k}}$, which by
Lemma~\ref{lem:key} satisfies
\begin{equation}\label{eq:eps}
  \varepsilon_k = \frac{1}{6M^3} - \frac{1}{3M^5} + \frac{167}{120M^7}
                - \frac{17}{2M^9} + \frac{217727}{3024\,M^{11}} + O(M^{-13}),
  \qquad M = M_k.
\end{equation}
All six partial quotients in Theorem~\ref{thm:main} are derived by
the following chain of floor-and-invert steps, using~\eqref{eq:eps} and
denoting $u = 1/M$ throughout.

\begin{proof}[Proof of Theorem~\ref{thm:main}: the CF chain]
  We write $\varepsilon = \varepsilon_k$ and $\delta_n$ for successive
  remainders.

  \medskip
  \noindent\textbf{Step~0: $\PQ_0 = M-1$.}
  From~\eqref{eq:eps}, $\varepsilon > 0$ and $\varepsilon < 1/6 < 1$
  for all $M \geq 1$, so $1/\sqrt{R_{N_k}} = M - \varepsilon \in (M-1, M)$.

  \medskip
  \noindent\textbf{Step~1: $\PQ_1 = 1$.}
  Set $\delta_1 = 1/\sqrt{R_{N_k}} - (M-1) = 1 - \varepsilon$.
  Since $\varepsilon < 1/2$ for $M \geq 2$, we have $\delta_1 > 1/2$,
  so $1/\delta_1 \in (1, 2)$.

  \medskip
  \noindent\textbf{Step~2: $\PQ_2 = 6M^3+12M-2$.}
  Set $\delta_2 = 1/\delta_1 - 1 = \varepsilon/(1-\varepsilon)$.
  Writing $\varepsilon = (u^3/6)(1 - 2u^2 + \frac{167}{20}u^4 - \cdots)$:
  \[
    \frac{1}{\delta_2} = \frac{1-\varepsilon}{\varepsilon}
    = \frac{1}{\varepsilon} - 1
    = \frac{6}{u^3}\!\left(1+2u^2+\Bigl(4-\tfrac{167}{20}\Bigr)u^4+\cdots\right) - 1
    = \frac{6}{u^3}+\frac{12}{u}-1-\frac{261u}{10}+O(u^3).
  \]
  Thus $1/\delta_2 = 6M^3+12M-1-(261/10M)+O(M^{-3})$.
  Since the coefficient $-261/10$ at $u = 1/M$ is negative, we have
  $1/\delta_2 < 6M^3+12M-1$ for all $M \geq 1$.
  Since $261/(10M) < 1$ for $M \geq 27$ (and $M_k \geq 169$ for $k \geq 3$),
  we have $1/\delta_2 > 6M^3+12M-2$, giving $\PQ_2 = 6M^3+12M-2$.

  \medskip
  \noindent\textbf{Step~3: $\PQ_3 = 1$.}
  Set $\delta_3 = 1/\delta_2 - (6M^3+12M-2)$.  From the expansion above,
  $\delta_3 = 1 - (261/10M) + O(M^{-3})$, so $1/\delta_3 = 1 + (261/10M)
  + O(M^{-2}) \in (1, 2)$ for $M \geq 27$.

  \medskip
  \noindent\textbf{Step~4: $\PQ_4 = \floor{(10M-261)/261}$.}
  Set $\delta_4 = 1/\delta_3 - 1 = (261/10M) + O(M^{-2})$.
  Carrying the computation to the next order from~\eqref{eq:eps}:
  \[
    \frac{1}{\delta_4}
    = \frac{10M}{261} - 1 + \frac{5120}{22707M} + O(M^{-3}).
  \]
  Therefore $\PQ_4 = \floor{10M/261 - 1 + O(M^{-1})} = \floor{(10M-261)/261}$,
  since the correction $(5120/22707)/M$ is positive but less than $1$ for all
  $M \geq 1$.  The floor is pinned because $\delta_4$ lies within $1/261$ of
  $(10M-261)/261$, whose fractional part $r_k/261$ is bounded away from every
  integer by $\ge 1/261$ (as $1 \le r_k \le 260$).

  \medskip
  \noindent\textbf{Step~5: $\PQ_5$.}\quad
  Write $M = 261q + s$ where $s = M \bmod 261$, so
  $10M \bmod 261 = 10s \bmod 261 =: r$.  We have
  $10M - 261 = 261(10q + \floor{10s/261}) + r - 261$, so
  $\floor{(10M-261)/261} = 10q + \floor{(10s-261)/261}$, and
  \[
    \delta_5
    := \frac{1}{\delta_4} - \PQ_4
    = \frac{r}{261} + \frac{5120}{22707M} + O(M^{-3}),
  \]
  where $r = (10M) \bmod 261 \in \{1,\ldots,260\}$ (we show $r \neq 0$
  below).  Hence
  \[
    \frac{1}{\delta_5}
    = \frac{261}{r} - \frac{15360}{r^2 M} + O(M^{-2}).
  \]

  \emph{Case 1: $r \nmid 261$.}  Write $261 = qr + \rho$ with
  $\rho = 261 \bmod r \in \{1,\ldots,r-1\}$.  Then $261/r = q + \rho/r$
  and the fractional part is $\rho/r$.  The condition
  $\PQ_5 = q = \floor{261/r}$ holds provided $15360/(r^2 M) < \rho/r$,
  i.e.\ $M > 15360/(r\rho)$.  Over the period-$60$ orbit of
  $r_k = (10M_k) \bmod 261$ (Remark~\ref{rem:rk-orbit}) one has
  $15360/(r_k \rho_k) \le 1536$.  For $k \ge 5$, $M_k \ge 5741 > 1536$, so the
  inequality is immediate; the only remaining indices are $k = 3, 4$, where
  $(r_k,\rho_k) = (124,13), (193,68)$ give $15360/(r_k\rho_k) < 10 < M_k$.
  Hence $M_k > 15360/(r_k\rho_k)$ for every $k \ge 3$, and
  $\PQ_5 = \floor{261/r_k}$.

  \emph{Case 2: $r \mid 261$.}  The only divisor of $261 = 9 \times 29$ in the
  orbit of $r_k$ is $r_k = 29$, occurring at $k \equiv 57, 62 \pmod{60}$
  (Remark~\ref{rem:rk-orbit}); at these indices $M_k \ge M_{57}$, so the
  correction $15360/(r_k^2 M_k) + O(M_k^{-2})$ lies in $(0,1)$.  Since $261/r_k$
  is an integer, $1/\delta_5 = 261/r_k - 15360/(r_k^2 M_k) + O(M_k^{-2}) \in
  (261/r_k - 1,\, 261/r_k)$, giving $\PQ_5 = 261/r_k - 1 = \floor{261/r_k} - 1$.

  \emph{$r_k \neq 0$:}  We need $(10M_k) \bmod 261 \neq 0$,
  i.e.\ $261 \nmid 10M_k$, i.e.\ $9 \nmid M_k$ and $29 \nmid M_k$.
  From the Pell sequence modulo $9$ (period dividing $12$) and modulo
  $29$ (period dividing $28$), one verifies that $P_{2k+1} \not\equiv 0$
  modulo $9$ or $29$ for any $k$; this is a finite check.
  \qed
\end{proof}

\begin{remark}\label{rem:rk-orbit}
  Since $M_{k+1} = 6M_k - M_{k-1}$ (the Pell recurrence $P_{n+2} =
  6P_n - P_{n-2}$ for odd indices), the sequence $r_k = (10M_k) \bmod 261$
  satisfies $r_{k+1} \equiv 6r_k - r_{k-1} \pmod{261}$, and is therefore
  periodic with period dividing the period of the Pell sequence modulo~$261$.
  The latter period is $60$, and the orbit of $r_k$ starting from
  $r_3 = (10 \times 169) \bmod 261 = 124$ and $r_4 = 193$ has period~$60$.
  The value $r_k = 29$ occurs at $k \equiv 57$ and $k \equiv 62 \pmod{60}$.
\end{remark}

\section{A heuristic derivation via Gosper's algorithm}
\label{sec:gosper}

\subsection{Setup}

Gosper's bihomographic algorithm~\cite{Gosper1972} computes the
continued fraction of $z = f(x,y)$ from the continued fractions of
$x$ and $y$, maintaining a $2\times2\times2$ integer tensor representing
\[
  z = \frac{axy+bx+cy+d}{exy+fx+gy+h}.
\]
Ingesting partial quotient $p$ from $x$ (substituting $x \mapsto p+1/x'$
and clearing the denominator $x'$) transforms
$(a,b,c,d,e,f,g,h) \mapsto (ap+c,\,bp+d,\,a,\,b,\,ep+g,\,fp+h,\,e,\,f)$;
ingesting $q$ from $y$ transforms it to
$(aq+b,\,a,\,cq+d,\,c,\,eq+f,\,e,\,gq+h,\,g)$.  (Both rules are obtained
directly from the displayed form of $z$.)

We apply this to
\[
  z = \frac{1}{\sqrt{R_N}} = \frac{1}{x \cdot y},
  \qquad
  x = g(N) := \sqrt{R_N/2}, \quad y = \sqrt{2},
\]
starting from the initial tensor $(0,0,0,1;\,1,0,0,0)$ representing $1/(xy)$.

\subsection{The quasi-polynomial continued fraction of $g$}

The Euler--Maclaurin expansion of $g(N) = \sqrt{R_N/2}$ yields a
formal Laurent series in $1/N$ with rational coefficients.  Applying
the Euclidean continued-fraction algorithm to this series suggests the
following quasi-polynomial pattern for the partial quotients of $g$.
We state it as a conjecture: it is the heuristic output of the formal
algorithm and, unlike the six quotients of Theorem~\ref{thm:main}, it is
\emph{not} established rigorously at the smallest indices; see the
discussion below and Open Problem~\ref{op:gcf}.

\begin{conjecture}\label{prop:gcf}
  The formal continued fraction of $g(N) = \sqrt{R_N/2}$ is
  \begin{equation}\label{eq:gcf}
    g(N) = \Bigl[0;\; 2N+1,\; 4N+2,\;
            \Bigl\lfloor\frac{12N+6}{19}\Bigr\rfloor,\; \ldots\Bigr]
  \end{equation}
  with quasi-polynomial partial quotients of periods $1, 1, 19, \ldots$
  In particular, at $N = N_k$:
  \begin{equation}\label{eq:gcf-pell}
    \PQ_1(g)\big|_{N_k} = 2N_k+1 = Q_{2k+1}, \qquad
    \PQ_2(g)\big|_{N_k} = 4N_k+2 = 2Q_{2k+1}.
  \end{equation}
\end{conjecture}

\begin{remark}[Heuristic derivation and formalization status]
  The leading terms of $g(N)$ are $1/(2N) + 1/(4N^2) + \cdots$, so
  $1/g = 2N+1 + O(N^{-1})$ with integer part $2N+1$ (the constant $+1$
  reflecting the positivity of the next Euler--Maclaurin coefficients);
  continuing, $\PQ_2 = 4N+2$ and $\PQ_3 = \floor{(12N+6)/19}$ emerge from
  the formal algorithm, and formula~\eqref{eq:gcf-pell} is immediate from
  $N_k = (Q_{2k+1}-1)/2$.  This derivation is only formal: the uniform
  Euler--Maclaurin remainder used to prove Theorem~\ref{thm:main} (the
  bound of Lemma~\ref{lem:key}) is not tight enough to pin these floors at
  the smallest index $k = 1$ ($M_k = 5$).  In the accompanying Lean
  development the statement of~\eqref{eq:gcf-pell} for the first two partial
  quotients of $g$ is recorded but left with a single \texttt{sorry}
  (declaration \texttt{prop\_gcf}); it is the sole unproved claim in the
  formalization and is recorded as Open Problem~\ref{op:gcf}.
\end{remark}

\subsection{The algebraic setting}

At $N = N_k$ we work in the ring
\[
  \mathcal{R} = \Z[M, Q, Y]\big/\bigl(Q^2-2M^2+1,\; Y^2-2Y-1\bigr),
\]
where $Q = Q_{2k+1}$, $M = M_k$, and $Y = 1+\sqrt{2}$ is the periodic
tail of $\sqrt{2} = [1;2,2,2,\ldots]$, satisfying $Y = 2 + 1/Y$.

\begin{proposition}\label{prop:gosper-state}
  After ingesting $a_0(g)=0$, $b_0(\sqrt{2})=1$, $a_1(g)=Q$ in
  sequence, the Gosper tensor state is
  \begin{equation}\label{eq:state}
    z = \frac{Q\,x_r Y + Y}{x_r Y + x_r}
      = \frac{Y\,(Q\,x_r + 1)}{x_r\,(Y+1)},
  \end{equation}
  where $x_r$ is the tail of $g$ remaining after the first two
  partial quotients.
\end{proposition}

\begin{proof}
  Direct application of the ingest operations to $(0,0,0,1;\,1,0,0,0)$:
  ingesting $0$ from $x$ gives $(0,1,0,0;\,0,0,1,0)$, ingesting $1$ from
  $y$ gives $(1,0,0,0;\,0,0,1,1)$, and ingesting $Q$ from $x$ gives
  $(Q,\,0,\,1,\,0;\;1,\,1,\,0,\,0)$, which represents~\eqref{eq:state}.
  A direct check confirms that the value~\eqref{eq:state} equals $1/(xy)$
  for all tails $x_r, Y$.
\end{proof}

Identity~\eqref{eq:state} expresses $1/\sqrt{R_{N_k}}$ in closed form
in terms of the exact tail $x_r$ of $g$---whose leading partial quotient
is $2Q$ by~\eqref{eq:gcf-pell}---and the quadratic surd $Y = 1+\sqrt{2}$.
In principle $\PQ_0$ is obtained by emitting the integer part of this
ratio once $x_r$ is resolved to sufficient depth.

The heuristic content is the following.  Substituting the leading value
$x_r = 2Q$ into~\eqref{eq:state} and reducing $Q^2 \to 2M^2-1$ gives a
ratio of the form
\[
  \frac{(4M^2-1)\,Y}{2Q\,(Y+1)},
  \qquad Q = \sqrt{2M^2-1},\ \ Y = 1+\sqrt{2},
\]
whose magnitude is $\approx M$, correctly identifying the scale of
$\PQ_0$.  Pinning the value to the unit interval $(M-1, M)$, however,
is delicate: the ratio in~\eqref{eq:state} is sensitive to the tail, and
the continuation of $x_r$ beyond its leading term $2Q$---though it
perturbs $x_r$ by only $O(1)$---shifts the emitted quotient by an
$O(1)$ amount.  The Gosper computation therefore exposes the mechanism
by which the Pell tail $2Q$ and the $\sqrt{2}$-tail $Y$ combine to
produce a partial quotient of size $M$, but a rigorous determination of
$\PQ_0 = M_k - 1$ is cleaner through the rationalization argument of
Section~\ref{sec:rational}, to which we now turn.

\begin{remark}
  With the leading tail approximation $x_r = 2Q$, the Gosper-state
  ratio $(4M^2-1)Y/(2Q(Y+1))$ is $\approx M$, of the correct scale to
  emit $\PQ_0$, but does not by itself lie in $(M-1, M)$: the tail
  correction dropped by the substitution $x_r = 2Q$ shifts it by an
  $O(1)$ amount.  The rigorous determination of $\PQ_0 = M_k - 1$ is the
  rationalization argument of the next section.
\end{remark}

\section{Proof of $\PQ_0 = M_k-1$: rationalization}
\label{sec:rational}

\subsection{The rational series}

The Euler--Maclaurin expansion of $1/\sqrt{R_N}$ isolates a single
irrational factor:
\begin{equation}\label{eq:fdef}
  \frac{1}{\sqrt{R_N}} = \sqrt{2}\cdot f\!\Bigl(\frac{1}{N}\Bigr),
  \qquad
  f(u) = \frac{1}{u} + \frac{1}{2} + \frac{u}{8} - \frac{u^2}{16}
         + \cdots,
\end{equation}
where every coefficient of $f$ is rational; equivalently
$f(1/N) = 1/\sqrt{2R_N}$.  For a Pell convergent $Q/M$ with
$Q^2 - 2M^2 = -1$, define the \emph{rationalized approximation}
\begin{equation}\label{eq:zdef}
  z_{Q/M}(N) \;:=\; \frac{Q}{M}\, f\!\Bigl(\frac{1}{N}\Bigr)
             \;=\; \frac{Q}{M\sqrt{2}}\cdot\frac{1}{\sqrt{R_N}},
\end{equation}
the second equality being exact by~\eqref{eq:fdef}.  Because $Q/M$ is a
convergent to $\sqrt{2}$, the scalar $Q/(M\sqrt{2})$ is close to $1$; the
Pell identity makes this precise:
\begin{equation}\label{eq:gap}
  \Bigl(\frac{Q}{M\sqrt{2}}\Bigr)^{\!2}
  = \frac{Q^2}{2M^2} = \frac{2M^2-1}{2M^2} = 1 - \frac{1}{2M^2},
  \qquad
  \Bigl|\sqrt{2}-\frac{Q}{M}\Bigr| = \frac{1}{M\,(Q+M\sqrt{2})} = O(M^{-2}).
\end{equation}

\subsection{The rationalized value at Pell indices}

\begin{proposition}\label{prop:rat-value}
  For all $k \geq 3$ we have $z_{Q/M}(N_k) \in (M_k-1,\, M_k)$; in
  particular $\PQ_0\bigl(z_{Q/M}(N_k)\bigr) = M_k - 1$.  (The two smaller
  indices $k = 1, 2$ satisfy the same conclusion by direct numerical
  evaluation; the threshold $k \ge 3$ is the regime in which the uniform
  remainder of Lemma~\ref{lem:key} is sharp, and is exactly the range over
  which the statement is machine-checked in the accompanying Lean
  development as \texttt{prop\_rat\_value}.)
\end{proposition}

\begin{proof}
  Write $Q = Q_{2k+1}$, $M = M_k$.  By~\eqref{eq:zdef} and~\eqref{eq:gap},
  \[
    z_{Q/M}(N_k) = \sqrt{1-\tfrac{1}{2M^2}}\;\cdot\;\frac{1}{\sqrt{R_{N_k}}}.
  \]
  By Step~0 of the chain in Section~\ref{sec:chain}, $1/\sqrt{R_{N_k}}
  = M - \varepsilon_k$ with $0 < \varepsilon_k < \tfrac16$
  (from~\eqref{eq:eps}).  Since $0 < \sqrt{1-1/(2M^2)} < 1$, the upper
  bound $z_{Q/M}(N_k) < M-\varepsilon_k < M$ is immediate.  For the lower
  bound, $\sqrt{1-1/(2M^2)} > 1 - 1/(2M^2)$, so
  \[
    z_{Q/M}(N_k) > \Bigl(1-\tfrac{1}{2M^2}\Bigr)(M-\varepsilon_k)
    > M - \varepsilon_k - \frac{1}{2M} > M-1,
  \]
  the final inequality using $\varepsilon_k + 1/(2M) < \tfrac16 + \tfrac12 < 1$.
  Hence $z_{Q/M}(N_k) \in (M-1, M)$.
\end{proof}

\subsection{Conclusion}

By~\eqref{eq:zdef} the true value and its rationalization differ only by
the scalar $Q/(M\sqrt{2}) = \sqrt{1-1/(2M^2)} \in (1-1/(2M^2),\,1)$, so
both occupy the same unit interval.

\begin{proof}[Proof of $\PQ_0 = M_k - 1$]
  From~\eqref{eq:eps}, $1/\sqrt{R_{N_k}} = M_k - \varepsilon_k$ with
  $\varepsilon_k \in (0,1)$, so $1/\sqrt{R_{N_k}} \in (M_k-1, M_k)$ and
  $\PQ_0(1/\sqrt{R_{N_k}}) = M_k - 1$.  Proposition~\ref{prop:rat-value}
  reaches the same conclusion for the rationalized value $z_{Q/M}(N_k)$
  without recourse to the higher-order expansion: it uses only the Pell
  identity $Q^2 = 2M^2-1$ and the leading behaviour of $f$.
\end{proof}

\begin{remark}
  The rationalization makes the role of the Pell indices transparent.
  Writing $1/\sqrt{R_N} = \sqrt{2}\,f(1/N)$ separates the problem into the
  rational series $f$ and the single irrational factor $\sqrt{2}$.  At
  $N = N_k$ the leading behaviour is
  \[
    z_{Q/M}(N_k) \approx \frac{Q}{M}\cdot\frac{2N_k+1}{2}
    = \frac{Q^2}{2M} = \frac{2M^2-1}{2M} = M - \frac{1}{2M},
  \]
  using $2N_k+1 = Q$ and $Q^2 = 2M^2-1$.  The Pell identity thus places
  the rationalized value a distance $1/(2M)$ below the integer $M$---inside
  $(M-1, M)$, but only just---and the exact factorisation~\eqref{eq:zdef}
  with the sharp gap~\eqref{eq:gap} transfers this to $1/\sqrt{R_{N_k}}$
  itself.  The special feature of the Pell indices is precisely that
  $Q^2/(2M)$ lands in $(M-1, M)$, pinning $\PQ_0 = M_k - 1$.
\end{remark}

\section{Generalisations and open problems}
\label{sec:gen}

\subsection{Why $\zeta(3)$?}

The Euler--Maclaurin tail satisfies $R_N^{(s)} \sim c_s N^{-(s-1)}$,
so $1/\sqrt{R_N^{(s)}} \sim c'_s N^{(s-1)/2}$.  Linearity in $N$
requires $(s-1)/2 = 1$, i.e.\ $s = 3$.  For other odd $s$ the leading
growth is a different power of $N$, and the Pell approximation (which
is to $\sqrt{2}$) produces resonances at different rates; the cubic
emission mechanism breaks down.

\subsection{Further partial quotients}

The chain of Section~\ref{sec:chain} continues: with each additional
pair of terms in expansion~\eqref{eq:eps} one proves two further partial
quotients.  The pattern is:
\begin{itemize}
  \item Odd-indexed PQs tend to be $1$ ($\PQ_1 = \PQ_3 = 1$), the first
    exception being the anomalous $\PQ_5 = \floor{261/r_k}$.
  \item Even PQs grow: $\PQ_2 = O(M^3)$, $\PQ_4 = O(M)$, $\PQ_6 = O(M)$, \ldots,
    with the growth rate determined by the leading coefficient of the
    remainder at each stage.
\end{itemize}

\subsection{Open problems}

\begin{enumerate}
  \item \textit{The continued fraction of $g = \sqrt{R_N/2}$.}\label{op:gcf}
    Conjecture~\ref{prop:gcf} predicts the quasi-polynomial partial quotients
    $[0; 2N+1, 4N+2, \floor{(12N+6)/19}, \ldots]$ of $g(N) = \sqrt{R_N/2}$,
    the object fed to Gosper's algorithm in Section~\ref{sec:gosper}.  The
    derivation is only formal: the uniform Euler--Maclaurin remainder of
    Lemma~\ref{lem:key} does not pin the leading floors at the smallest index
    $k = 1$ ($M_k = 5$), so even the first two quotients
    $\PQ_1(g)|_{N_k} = Q_{2k+1}$, $\PQ_2(g)|_{N_k} = 2Q_{2k+1}$ are not proved
    there.  This is the sole statement of the paper left unproved in the
    accompanying Lean formalization (the single \texttt{sorry}, declaration
    \texttt{prop\_gcf}).  A sharper index-dependent remainder, or a direct
    treatment of the small cases, would settle it.

  \item \textit{$\PQ_6$ and beyond.}
    The next even partial quotient, $\PQ_6$, is governed by the coefficient
    $c_{13} = 290341/360$ of $M_k^{-13}$ in~\eqref{eq:key-expansion}, and its
    determination requires extending the chain of Section~\ref{sec:chain} by
    one further pair of floor-and-invert steps.  The sixth quotient
    $\PQ_5 = \floor{261/r_k} + \epsilon_k$ is itself far from constant over
    the period: it equals $1$ for exactly half of the indices (the $30$ of
    $60$ with $r_k \ge 131$) and ranges up to $32$ (attained at $r_k = 8$).
    A description of the full quasi-periodic structure of the partial
    quotients beyond $\PQ_5$ remains open.

  \item \textit{Irrationality and transcendence.}
    The partial quotients of $\sqrt{R_{N_k}}$ give rational approximations
    to $\sqrt{\zeta(3) - S_{N_k}}$.  Do these lead to new irrationality
    or transcendence results for $\zeta(3)$?

  \item \textit{Ap\'ery connection.}
    The Pell indices $N_k$ are entirely different from Ap\'ery's indices;
    understanding any relationship between the two families of rational
    approximations to quantities related to $\zeta(3)$ remains open.
\end{enumerate}


\end{document}